\documentclass[10pt]{article}
\usepackage{latexsym,amssymb,amsmath}
\usepackage{amssymb,amsthm,upref,amscd}
\usepackage{color}
\usepackage[T1]{fontenc}
\begin{document}
\baselineskip=13pt

\numberwithin{equation}{section}

\newtheorem{thm}{Theorem}[section]
\newtheorem{lem}[thm]{Lemma}
\newtheorem{cor}[thm]{Corollary}
\newtheorem{Prop}[thm]{Proposition}
\newtheorem{Def}[thm]{Definition}
\newtheorem{Rem}[thm]{Remark}
\newtheorem{Ex}[thm]{Example}

\newcommand{\A}{\mathbb{A}}
\newcommand{\B}{\mathbb{B}}
\newcommand{\C}{\mathbb{C}}
\newcommand{\D}{\mathbb{D}}
\newcommand{\E}{\mathbb{E}}
\newcommand{\F}{\mathbb{F}}
\newcommand{\G}{\mathbb{G}}
\newcommand{\I}{\mathbb{I}}
\newcommand{\J}{\mathbb{J}}
\newcommand{\K}{\mathbb{K}}
\newcommand{\M}{\mathbb{M}}
\newcommand{\N}{\mathbb{N}}
\newcommand{\Q}{\mathbb{Q}}
\newcommand{\R}{\mathbb{R}}
\newcommand{\T}{\mathbb{T}}
\newcommand{\U}{\mathbb{U}}
\newcommand{\V}{\mathbb{V}}
\newcommand{\W}{\mathbb{W}}
\newcommand{\X}{\mathbb{X}}
\newcommand{\Y}{\mathbb{Y}}
\newcommand{\Z}{\mathbb{Z}}
\newcommand\ca{\mathcal{A}}
\newcommand\cb{\mathcal{B}}
\newcommand\cc{\mathcal{C}}
\newcommand\cd{\mathcal{D}}
\newcommand\ce{\mathcal{E}}
\newcommand\cf{\mathcal{F}}
\newcommand\cg{\mathcal{G}}
\newcommand\ch{\mathcal{H}}
\newcommand\ci{\mathcal{I}}
\newcommand\cj{\mathcal{J}}
\newcommand\ck{\mathcal{K}}
\newcommand\cl{\mathcal{L}}
\newcommand\cm{\mathcal{M}}
\newcommand\cn{\mathcal{N}}
\newcommand\co{\mathcal{O}}
\newcommand\cp{\mathcal{P}}
\newcommand\cq{\mathcal{Q}}
\newcommand\rr{\mathcal{R}}
\newcommand\cs{\mathcal{S}}
\newcommand\ct{\mathcal{T}}
\newcommand\cu{\mathcal{U}}
\newcommand\cv{\mathcal{V}}
\newcommand\cw{\mathcal{W}}
\newcommand\cx{\mathcal{X}}
\newcommand\ocd{\overline{\cd}}

\def\c{\centerline}
\def\ov{\overline}
\def\emp {\emptyset}
\def\pa {\partial}
\def\bl{\setminus}
\def\op{\oplus}
\def\sbt{\subset}
\def\un{\underline}
\def\al {\alpha}
\def\bt {\beta}
\def\de {\delta}
\def\Ga {\Gamma}
\def\ga {\gamma}
\def\lm {\lambda}
\def\Lam {\Lambda}
\def\om {\omega}
\def\Om {\Omega}
\def\sa {\sigma}
\def\vr {\varepsilon}
\def\va {\varphi}

\title{\bf Existence of solutions for a nonlinear Choquard equation with potential vanishing at infinity}

\author{Claudianor O. Alves\thanks{Partially supported by CNPq/Brazil
301807/2013-2, coalves@dme.ufcg.edu.br}\\
{\small  Universidade Federal de Campina Grande} \\ {\small Departamento de Matem\'{a}tica} \\ {\small CEP: 58429-900, Campina Grande - Pb, Brazil}
\\
\\
Giovany M. Figueiredo\thanks{Partially supported by CNPq/Brazil
301292/2011-9 and 552101/2011-7,  giovany@ufpa.br}
\\
{\small  Universidade Federal do Par\'a} \\ {\small Departamento de Matem\'atica} \\ {\small CEP: 66075-110 Bel\'em - Pa, Brazil}
\\
\\
Minbo Yang\thanks{Supported by NSFC (11571317, 11101374, 11271331) and ZJNSF(LY15A010010), mbyang@zjnu.edu.cn}\vspace{2mm}
\\
{\small  Department of Mathematics, Zhejiang Normal University} \\ {\small  Jinhua, Zhejiang, 321004, P. R. China.}}

\date{}
\maketitle

\begin{abstract}
We study the following class of nonlinear Choquard equation,
$$
-\Delta u +V(x)u =\Big( \frac{1}{|x|^\mu}\ast F(u)\Big)f(u)  \quad \mbox{in} \quad \R^N,
$$
where $0<\mu<N$, $N \geq 3$,   $V$ is a continuous real function and $F$ is the primitive function of $f$. Under some suitable assumptions on the potential $V$, which include the case $V(\infty)=0$, that is, $V(x)\to 0$ as $|x|\to +\infty$, we prove existence of a nontrivial solution for the above equation by penalization method.
 \vspace{0.3cm}


\vspace{0.3cm}

\end{abstract}

\section{Introduction and main results}
In this paper, we study the existence of nontrivial solutions for the following nonlinear Choquard equation
$$
\left\{\aligned &-\Delta u +V(x)u =\Big( \frac{1}{|x|^\mu}\ast F(u)\Big)f(u)  \quad \mbox{in} \quad \R^N, \\
&  u\in D^{1,2}({\R}^N),\\
& u(x)>0\ \ \mbox{for all}\ \ x\in\mathbb{R}^N,
\endaligned\right.\eqno{(SNE)}
$$
where $0<\mu< N$, $N \geq 3$, $\Delta$ is the Laplacian operator, $V$ is a nonnegative continuous real function and $F$ is the primitive function of $f$.

This problem comes from looking for standing wave solutions for a nonlinear Schr\"{o}dinger
equation of the kind
\begin{equation}\label{EN}
i\partial_{t} \Psi =-\Delta \Psi
+W(x)\Psi-(Q(x)\ast|\Psi|^{q})|\Psi|^{q-2}\Psi,\ \ \
 \quad \mbox{in} \quad \R^N.
\end{equation}
Here $W$ is the external potential and $Q$ is the
response function possesses information on the mutual interaction
between the bosons. This type of nonlocal equation is known to describe the propagation of electromagnetic
waves in plasmas \cite{BC} and plays an important
role in the theory of Bose-Einstein condensation \cite{D}. It is clear that $\Psi(x,t)=u(x)e^{{-iE}t}$ solves
the evolution equation \eqref{EN} if, and only if, $u$ solves
\begin{equation}\label{SN}
 -\Delta u +V(x)u  =(Q(x)\ast|u|^{q})|u|^{q-2}u  \quad \mbox{in} \quad \R^N,
\end{equation}
with $V(x)=W(x)-E$.

When the response function is the Dirac function, i.e. $Q(x)=\delta(x)$, the
nonlinear response is local indeed and we are lead to  the Schr\"odinger equation
\begin{equation}\label{S.S}
 -\Delta u +V(x)u  = |u|^{q-2}u  \quad \mbox{in} \quad \R^N,
\end{equation}
which has been studied
extensively under various hypotheses on the potentials and the nonlinearities \cite{ASM, AS, AFM, BPW, BSc, BW, BP, DF2, DF1, DT, DS, FW, JT, LWZ}.

The aim of this paper is to study the existence of solutions for a class of nonlocal Schr\"odinger equation, that is, the response function $Q$ in $(SNE)$ is of Coulomb type, for example $\frac{1}{|x|^{\mu}}$, then we arrive at the Choquard-Pekar equation,
\begin{equation}\label{Nonlocal.S1}
 -\Delta u +V(x)u =\Big(\frac{1}{|x|^{\mu}}\ast|u|^{q}\Big)|u|^{q-2}u  \quad \mbox{in} \quad \R^N.
\end{equation}
The case $q=2$ and $\mu=1$ goes
back to the description of the quantum theory of a polaron at rest by S. Pekar in 1954 \cite{P1}
and the modeling of an electron trapped
in its own hole in 1976 in the work of P. Choquard, in a certain approximation to Hartree-Fock theory of one-component
plasma \cite{L1}. The equation is also known as the Schr\"{o}dinger-Newton equation, which was introduced by Penrose in his discussion on the selfgravitational collapse of a quantum mechanical wave-function.  Penrose suggested that the solutions of \eqref{Nonlocal.S1}, up to reparametrization, are the basic stationary states which do not spontaneously collapse any further, within a certain time scale.

 As far as we know, most of the existing papers consider the existence and property of the solutions for the nonlinear Choquard equation $(SNE)$ with constant potential. In \cite{L1}, Lieb proved the existence and uniqueness, up to translations,
of the ground state to equation \eqref{Nonlocal.S1}. Later, in \cite{Ls}, Lions showed the existence of
a sequence of radially symmetric solutions to this equation. Involving the properties of the ground state solutions,
 Ma and Zhao \cite{ML} considered the generalized Choquard equation \eqref{Nonlocal.S1} for $q\geq 2$, and they proved that every positive solution of (\ref{Nonlocal.S1}) is radially symmetric and monotone decreasing about some point, under the assumption that a certain
set of real numbers, defined in terms of $N, \alpha$ and $q$, is nonempty. Under the
same assumption, Cingolani, Clapp and Secchi \cite{CCS1}  gave some existence and
multiplicity results in the electromagnetic case, and established the regularity and
some decay asymptotically at infinity of the ground states of \eqref{Nonlocal.S1}. In \cite{MS1}, Moroz and Van
Schaftingen  eliminated this restriction and showed the regularity, positivity
and radial symmetry of the ground states for the optimal range of parameters, and
derived decay asymptotically at infinity for them as well. Moreover, Moroz and Van
Schaftingen in \cite{MS2, MS4} also considered  the existence of ground states for critical growth in the sense of Hardy-Littlewood-Sobolev inequality and under the assumption of Berestycki-Lions type.

Involving the problem with nonconstant potentials
\begin{equation}\label{Nonlocal.S2}
 -\Delta u +V(x)u =\Big(\frac{1}{|x|^{\mu}}\ast F(u)\Big)f(u)  \quad \mbox{in} \quad \R^N,
\end{equation}
where $V$ is a continuous periodic function with $\inf_{\mathbb{R}^{N}} V(x)> 0$, noticing that the nonlocal term is invariant under translation, it is possible to prove an existence result easily by applying the Mountain Pass Theorem, see \cite{AC} for example.
For a periodic potential $V$ that changes sign and $0$ lies in the gap of the spectrum of the Schr\"{o}dinger operator $-\Delta +V$, the  problem is strongly indefinite, and in \cite{BJS}, the authors proved the existence of nontrivial solution with $\mu=1$ and $F(u)=u^2$ by reduction methods. For a general class of response function $
Q$ and nonlinearity $f$, Ackermann \cite{AC} proposed an approach to prove the
existence of infinitely many geometrically distinct weak solutions.

As far as we know, the only papers that considered the generalized Choquard equation with vanishing potential are \cite{MS3, S}. Replacing $\Delta$ by $\vr^2\Delta$, S. Secchi \cite{S} obtained the existence result for small $\vr$ when $V>0$ and $\liminf_{|x|\to\infty}V (x)|x|^\gamma > 0$ for some
$\gamma \in [0, 1)$  by a Lyapunov-Schmidt type reduction. In \cite{MS3}, the authors prove the existence and concentration behavior of the semiclassical states for the problem with $F(u)=u^p$ for small $\vr$. There, they developed a penalization technique by constructing supersolutions to a linearization of the penalized problem in an outer domain, and after,  estimate the
solutions of the penalized problem by some comparison principle. Besides these two papers, in our mind, for constant $\vr=1$ and general nonlinearities, the existence of nontrivial solution for the Choquard equation with vanishing potential is still not known.

Motivated by the above facts and a recent paper due to Alves and Souto \cite{AS}, in the present paper we intend to study the Choquard equation $(SNE)$ with potentials vanishing at infinity, that is, $V(\infty)=0$. This class of problems goes back to the work by Berestycki and Lions in \cite{BP} where the authors studied the Schr\"odinger equation with zero mass and showed that the problem
$$
\left\{
\begin{array}{l}
-\Delta u = f(u), \mbox{in} \,\,\ \   \mathbb{R}^{N}, \\
u \in D^{1,2}(\mathbb{R}^{N}),
\end{array}
\right.
$$
has no ground state solution if $f(t)=|t|^{p-2}t$. However, they also proved that if $f$ behaves like $|t|^{q-2}t$ for small $t$ and $|t|^{p-2}t$ for large $t$ when $p<2^*<q$, then the problem has a ground state solution. For example, these conditions are verified by the function
$$
f(t)=
\left\{
\begin{array}{l}
|t|^{q-2}t, \,\, t \leq 1,\\
h(t), \,\, 1 \leq t \leq 2,\\
|t|^{p-2}t, \,\,\, t \geq 2,
\end{array}
\right.
$$
where $h$ is selected so that $f$ is a $C^{1}(\mathbb{R})$ function. Later Benci, Grisanti and Micheletti \cite{BGM1,BGM2} considered problems of the type
$$
\left\{
\begin{array}{l}
-\Delta u + V(x)u= f(u), \mbox{in} \,\,  \ \ \mathbb{R}^{N}, \\
u \in D^{1,2}(\mathbb{R}^{N})
\end{array}
\right.
$$
and studied the existence and nonexistence of ground state solution.

Motivated by the above papers and their hypotheses, we assume the following conditions on the nonlinearity $f$:
$$
\lim_{s\to 0^{+}}\frac{sf(s)}{s^q}<+\infty,\ \ \eqno{(f_1)}
$$
$q\geq 2^*=\frac{2N}{N-2}$. Assume that $0<\mu<\frac{N+2}{2}$ and
$$
\lim_{s\to +\infty}\frac{sf(s)}{s^{p}}=0, \eqno{(f_2)}
$$
 for some $p \in  (1, \frac{2(N-\mu)}{N-2})$.  Note that interval $(1, \frac{2(N-\mu)}{N-2})$ is not empty because of the above condition on $\mu$. We would like point out that by $(f_1)$, since $\frac{2(N-\mu)}{N-2}<\frac{2N-\mu}{N-2}<2^*$, there holds
\begin{equation} \label{ZR1}
\lim_{s\to0}\frac{sf(s)}{s^{\frac{2N-\mu}{N-2}}}=0\ \ \hbox{and}\ \ \lim_{s\to0}\frac{sf(s)}{s^p}=0.
\end{equation}
And thus, by $(f_1)-(f_2)$ and (\ref{ZR1}), there exists $c_0>0$ such that
\begin{equation} \label{Growth}
|sf(s)|\leq c_0|s|^{2^*},\ \ \ |sf(s)|\leq c_0|s|^{q}, \ \ \ |sf(s)|\leq c_0|s|^{\frac{2N-\mu}{N-2}} \ \  \hbox{and}\ \ |sf(s)|\leq c_0|s|^{p} \quad \forall s \in \mathbb{R}.
\end{equation}

The nonlinearity $f$ is also supposed to verify the Ambrosetti-Rabinowitz type
superlinear condition for nonlocal problem, that is, there exists $\theta>2$, such that
$$
0<\theta F(s)\leq 2f(s)s \,\,\,\, \forall s>0, \eqno{(f_3)}
$$
where $F(t)=\int^t_0f(s)ds$. Without loss of generality, we will assume that $\theta<4$ in the rest of the paper.

Related to the potential $V(x)$, we assume that it is a nonnegative continuous function, and we fix the following notations: \\
\begin{equation} \label{m}
m=\max_{|x|\leq 1}V(x)
\end{equation}
and we define the function $\mathcal{V}:(1,+\infty) \to [0, \infty)$ by
\begin{equation} \label{VR1}
\mathcal{V}(R)=\frac{1}{R^{(q-2)(N-2)}}\inf_{|x|\geq R}|x|^{(q-2)(N-2)}V(x).
\end{equation}
\\

One of the main results of this paper is the following existence result,
\begin{thm}\label{MRS}
Assume that $0<\mu<\frac{N+2}{2}$ and conditions $(f_1)-(f_3)$ hold. There exists a constant  $\mathcal{V}_0=\mathcal{V}_0(m,\theta,p,\mu,c_0)$ such that, if $\mathcal{V}(R) >\mathcal{V}_0$ for some $R>1$, then problem $(SNE)$ has a positive solution.
\end{thm}

The constant $\mathcal{V}_0$ is a technical constant, which will appear in the proof of the main result, see Section 3. Next, we will introduce an example of potential $V$ for Theorem \ref{MRS}. \\

\noindent {\bf Example 1:} Let
$$
V(x)=
\left\{
\begin{array}{l}
\phi(x), \quad \mbox{if} \quad |x|\leq 2, \\
\mbox{}\\
\mathcal{V}_02^{(q-2)(N-2)+1}|x|^{-(q-2)(N-2)}, \quad \mbox{if} \quad |x|\geq 2, \\
\end{array}
\right.
$$
where $\phi:\overline{B}_2(0) \to [0,+\infty)$ is a continuous, which makes continuous $V$. Then we know that
$$
\mathcal{V}(2)= 2 \mathcal{V}_0,
$$
and so,
$$
\mathcal{V}(2) > \mathcal{V}_0.
$$

If the potential $V$ is a radial function, that is,
$$
V(x)=V(|x|), \,\,\, \forall x \in \mathbb{R}^{N},
$$
we define the function $\mathcal{W}:(1,+\infty) \to [0, \infty)$ by
\begin{equation} \label{WR1}
\mathcal{W}(R)=\inf_{|x|\geq R}|x|^{\frac{4-\mu}{2}}V(x)
\end{equation}
and we can apply the arguments in Theorem \ref{MRS} to treat the homogeneous case with nonlinearity
$$
f(t)=|t|^{\frac{4-\mu}{N-2}}t, \,\,\,\,\,\, \forall t \in \mathbb{R}, \eqno{(f_4)}
$$
where $0<\mu < \min\{\frac{N+2}{2}, 4\}$.\\

Involving this situation, we have the following existence result.
\begin{thm}\label{MRS2}
Assume that $0<\mu < \min\{\frac{N+2}{2}, 4\}$, $V$ is a radial function and condition $(f_4)$ holds. There exists a constant  $\mathcal{W}_0=\mathcal{W}_0(m,\mu)$ such that, if $\mathcal{W}(R)>\mathcal{W}_0$ for some $R>1$, then problem $(SNE)$ has a positive solution.
\end{thm}

As an application, we can use Theorem \ref{MRS2} to study
$$
\left\{
\begin{array}{l}
-\Delta u +V(x)u =\Big( \frac{1}{|x|}\ast u^{5}\Big)u^{4} \quad \mbox{in} \quad \R^3,
\mbox{}\\
u(x)>0 \,\,\, \mbox{in} \,\,\, \mathbb{R}^{3}.
\end{array}
\right.
$$
where $V$ is a potential as below: \\

\noindent {\bf Example 2:} Let $\delta >0$ and
$$
V(x)=
\left\{
\begin{array}{l}
1, \quad \mbox{if} \quad |x|\leq 1, \\
\mbox{}\\
|x|^{-\frac{(4-\mu)}{2}+\delta}, \quad \mbox{if} \quad |x|\geq 1. \\
\end{array}
\right.
$$
We have that
$$
\mathcal{W}(R)= R^{\delta}, \quad \forall R>1.
$$
Hence,
$$
\lim_{R \to +\infty}\mathcal{W}(R)=+\infty,
$$
and so, for $R$  large enough
$$
\mathcal{W}(R) > \mathcal{W}_0.
$$

\noindent {\bf Example 3:} Let
$$
V(x)=
\left\{
\begin{array}{l}
2\mathcal{W}_0, \quad \mbox{if} \quad |x|\leq 1, \\
\mbox{}\\
2\mathcal{W}_0|x|^{-\frac{(4-\mu)}{2}}, \quad \mbox{if} \quad |x|\geq 1. \\
\end{array}
\right.
$$
where $\phi:\overline{B}_1(0) \to [0,+\infty)$ is a continuous, which makes continuous $V$. For the above potential, we have that
$$
\mathcal{W}(R)= 2\mathcal{W}_0, \quad \forall R>1,
$$
and so,
$$
\mathcal{W}(R) > \mathcal{W}_0, \quad \forall R>1.
$$

We now briefly outline the organization of the contents of this paper. In Section 2, we introduce an auxiliary problem, by using a penalization method due to del Pino \& Felmer \cite{DF1}. Section 3 is devoted to prove Theorem \ref{MRS}. Finally in Section 4, we consider the case where the potential $V$ is a radial function and the nonlinearity is homogeneous. By adjusting some arguments in the proof of Theorem \ref{MRS}, we can prove the result in Theorem \ref{MRS2}.

\vspace{0.5 cm}

{\bf Notations} \\

We fix the following notations, which will be used from now on. \\

\noindent $\bullet$ $C$, $C_i$ denote positive constants.\\
\noindent $\bullet$ $B_R$ denote the open ball centered at the origin with
radius $R>0$. \\
\noindent $\bullet$ $L^s(\R^N)$, for $1 \leq s \leq \infty$,
denotes the Lebesgue space with the norms
$$
| u |_s:=\Big(\int_{\R^N}|u|^s\Big)^{1/s}.
$$
\noindent $\bullet$ If $u$ is a mensurable function, we denote by $u_+$ and $u_-$ its positive and negative part respectively, i.e.,
$$
u_+(x)=\max\{u(x),0\} \quad \mbox{and} \quad u_-(x)=\max\{-u(x),0\}.
$$

\noindent $\bullet$ $C_0^{\infty}(\R^N)$ denotes  the space of the functions
infinitely differentiable with compact support in $\R^N$. \\
\noindent $\bullet$ We denote by $D^{1,2}(\R^N)$  the Hilbert space
$$
D^{1,2}(\R^N)=\{u \in L^{2^{*}}(\mathbb{R}^{N})\,:\, \nabla u \in L^{2}(\mathbb{R^{N}}) \}
$$
endowed with the inner product
$$
\left\langle u,v  \right\rangle = \int_{\mathbb{R^{N}}}\nabla u \nabla v
$$
and norm
$$
    \|u\|_{D^{1,2}}:=\left(\int_{\R^N}|\nabla u|^2\right)^{1/2}.
$$

\noindent $\bullet$ $S$ is the best constant that verifies
$$
|u|^2_{2^{*}}\leq S\int_{\R^{N}}|\nabla u|^{2},\ \ \forall u\in D^{1,2}(\R^N).
$$
\noindent $\bullet$  From the assumptions on $V$, it follows that the subspace
  $$
  E=\{u\in D^{1,2}(\R^N):
  \ \int_{\R^N} V( x)|u|^2<\infty\}
  $$
 is a Hilbert space with norm defined by
 $$
 \|u\|:=\left(\int_{\R^N}(|\nabla u|^2+ V( x)|u|^2)\right)^{1/2}.
 $$

\noindent $\bullet$ \, Let $E$ be a real Hilbert space and $I:E \to \R$ be a functional of class $\mathcal{C}^1$.
 We say that $(u_n)\subset E$ is a  Palais-Smale  sequence at level $c$ for $I$, $(PS)_c$ sequence for short, if $(u_n)$ satisfies
$$
I(u_n)\to c \,\,\, \mbox{and} \,\,\,\, I'(u_n)\to0,  \,\,\, \mbox{as} \,\,\, n\to\infty.
$$
Moreover, we say that $I$ satisfies the $(PS)_c$ condition, if any $(PS)_c$ sequence possesses a convergent subsequence. \\

\section{The penalized problem}

In this section, adapting some arguments introduced by del Pino and Felmer in \cite{DF1}, we are able to obtain the existence of nontrivial solution for $(SNE)$ by studying an auxiliary problem.

We would like to make some comments on the assumptions involving the nonlinearity $f$. The first point is that since we are looking for positive solutions,  we may suppose that
$$
f(s) = 0, \quad \forall\ s < 0.
$$
The second one is due to the application of variational methods, in this way, we must have
\begin{equation} \label{Z1}
\left|\int_{\R^N}\big(\frac{1}{|x|^\mu}\ast F(u)\big) F(u)\right| < \infty, \quad \forall\ u \in E.
\end{equation}
To observe the above integrability, it is important to recall the Hardy-Littlewood-Sobolev inequality, which will be frequently used in the paper.

\begin{Prop} \cite[Theorem 4.3]{LL}$\,\,[Hardy-Littlewood-Sobolev \ inequality]$:\\
Let $s, r>1$ and $0<\mu<N$ with $1/s+\mu/N+1/r=2$. If $g\in
L^s(\R^N)$ and $h\in L^r(\R^N)$, then there exists a sharp constant
$C(s,N,\mu,r)$, independent of $g,h$, such that
$$
\int_{\R^N}\int_{\R^N}\frac{g(x)h(y)}{|x-y|^\mu}\leq
C(s,N,\mu,r) |g|_s|h|_r.
$$
\end{Prop}

The above inequality guarantees that (\ref{Z1}) holds, because by $(f_3)$ and (\ref{Growth}),
\begin{equation} \label{Z2}
|F(u)| \leq 2c_0 |u|^{\frac{2N-\mu}{N-2}}, \quad \forall\ u \in E.
\end{equation}
Thereby, by Hardy-Littlewood-Sobolev inequality,
$$
\int_{\R^N}\big(\frac{1}{|x|^\mu}\ast |F(u)|\big) |F(u)|< +\infty
$$
if $F(u)\in L^t(\R^N)$ for $t>1$ and
$$
\frac2t+\frac{\mu}{N}=2,
$$
that is,
$$
t= \frac{2N}{2N - \mu}.
$$
However, by (\ref{Z2}),
$$
\int_{\mathbb{R^{N}}}|F(u)|^{t} \leq (2c_{0})^{t} \int_{\mathbb{R}^{N}}|u|^{2^{*}} < \infty, \quad \forall\ u\ \in E,
$$
shows that $F(u) \in L^{t}(\mathbb{R}^{N})$ for all $ u \in E$.

From the above commentaries, the Euler-Lagrange functional $I:E \to \mathbb{R}$ associated to $(SNE)$ given by
$$
I(u)=\frac12\|u\|^2-\frac12\int_{\R^N}\big(\frac{1}{|x|^\mu}\ast F(u)\big) F(u)
$$
is well defined and belongs to $\mathcal{C}^1(E,\R)$
with its derivative given by
$$
I'(u)\varphi=\int_{\mathbb{R}^N} (\nabla u\nabla \varphi+ V(x)u\varphi)-\int_{\R^N}\Big(  \frac{1}{|x|^\mu}\ast F(u)\Big)f(u)\varphi, \,\,\, \forall\ u, \varphi \in E.
$$
Thus, it is easy to see that the solutions of $(SNE)$ correspond to critical points of the energy functional $I$. However, whereas $I$ does not verify in  general the $(PS)$ condition, there are some difficulties to prove the existence of nontrivial critical points for it.

In order to overcome the lack of compactness of $I$, we adapt the penalization method introduced by del Pino and Felmer in \cite{DF1}. For $\ell>1$ and $R>1$ to be determined later, we set the functions
$$
\displaystyle \hat{f}(x,s):=
\left\{\begin{array}{l}
\displaystyle f(s)\ \ \ \mbox{if}\ \ \ \ell f(s)\leq V(x)s,\vspace{4mm}\\
\displaystyle \frac{V(x)}{\ell}s \ \  \mbox{if} \ \ \ \ell f(s)> V(x)s
\end{array}
\right.
$$
and
\begin{equation} \label{PNT}
g(x,s):=\mathcal{X}_{B_R}(x)f(s)+(1-\mathcal{X}_{B_R}(x))\hat{f}(x,s),
\end{equation}
where $\mathcal{X}_{B_R}$ denote the characteristic function of the ball $B_R$. Using the previous notations, let us introduce the auxiliary  problem
$$
\left\{
\begin{array}{l}
\displaystyle -\Delta u +V(x)u  =\Big(\frac{1}{|x|^\mu}\ast G(x, u)\Big)g( x, u) \,\,\,\, \mbox{in} \,\,\, \R^{N}, \\
\mbox{}\\
u\in D^{1,2}({\R}^N),
\end{array}
\right.
\eqno{(APE)}
$$
where $G(x,s):=\int_{0}^{s}g(x,\tau)d\tau$. A direct computation shows that, for all $s\in \R$, the following inequalities hold
\begin{equation}\label{PN}
\hat{f}(x,s)\leq f(s)\ \  \hbox{for all}\ \ x\in \R^N,
\end{equation}
$$
g(x,s)\leq \frac{V(x)}{\ell}s \ \  \hbox{for all}\ \ |x|\geq R,
$$

$$
G(x,s)=F(s) \ \  \hbox{for all}\ \ |x|\leq R,
$$
$$
G(x,s)\leq \frac{V(x)}{2\ell}s^2 \ \  \hbox{for all}\ \ |x|\geq R.
$$

\begin{Rem}\label{SN}
 It is easy to check that, if $u$ is a positive solution of the equation $(APE)$ with $ f(u(x))\leq V(x)u(x)$ for all $|x|\geq R$, then $g(x,u)=f(u)$, and therefore,
$u$ is indeed a solution of problem $(SNE)$.
\end{Rem}
The Euler-Lagrange functional associated to $(APE)$ is given by
$$
\Phi(u)=\frac12\|u\|^2-\Psi(u)
$$
where
$$
\Psi(u)=\frac12\int_{\R^N}\big(\frac{1}{|x|^\mu}\ast G(x,u)\big) G(x,u).
$$
 From (\ref{Growth}) and \eqref{PN}, we know that $\Phi$ is well defined and belongs to $\mathcal{C}^1(E,\R)$
with its derivative given by
$$
\Phi'(u)\varphi=\int_{\R^N} (\nabla u\nabla \varphi+ V(x)u\varphi)-\Psi'(u)\varphi, \ \ \ \ \forall u, \varphi \in E.
$$
Therefore, it is easy to see that the solutions of $(APE)$ correspond to the critical points of the energy functional $\Phi$.
\begin{lem}\label{AR2}
Assume that condition $(f_3)$ holds. Then,
$$
\Psi'(u)u\geq\theta\Psi(u)>0,\ \ \  \forall u \in E\backslash{\{0\}}.
$$
\end{lem}
\begin{proof}
For all $u \in E\backslash{\{0\}}$, a direct computation gives
$$
\aligned
\frac1\theta&\Psi'(u)u-\Psi(u)\\
&=\int_{|x|\leq R}\big(\frac{1}{|x|^\mu}\ast G(x,u)\big) \big(\frac1\theta f(u)u-\frac12F(u)\big) +\int_{|x|> R}\big(\frac{1}{|x|^\mu}\ast G(x,u)\big) \big(\frac1\theta g(x,u)u-\frac12G(x,u)\big)\\
&\geq \int_{\{|x|> R\}\cap \{f(u(x))\leq V(x)u(x)\}}\big(\frac{1}{|x|^\mu}\ast G(x,u)\big) \big(\frac1\theta g(x,u)u-\frac12G(x,u)\big)\\
&\hspace{1cm}+\int_{\{|x|> R\}\cap \{f(u(x))\geq V(x)u(x)\}}\big(\frac{1}{|x|^\mu}\ast G(x,u)\big) \big(\frac1\theta g(x,u)u-\frac12G(x,u)\big)\\
&=\int_{\{|x|> R\}\cap \{f(u(x))\leq V(x)u(x)\}}\big(\frac{1}{|x|^\mu}\ast G(x,u)\big) \big(\frac1\theta f(u)u-\frac12F(u)\big)\\
&\hspace{1cm}+\int_{\{|x|> R\}\cap \{f(u(x))\geq V(x)u(x)\}}\big(\frac{1}{|x|^\mu}\ast G(x,u)\big) \big(\frac1\theta g(x,u)u-\frac12G(x,u)\big)\\
&\geq\int_{\{|x|> R\}\cap \{f(u(x))\geq V(x)u(x)\}}\big(\frac{1}{|x|^\mu}\ast G(x,u)\big) \big(\frac1\theta g(x,u)u-\frac12G(x,u)\big)\\
&\geq\int_{\{|x|> R\}\cap \{f(u(x))\geq V(x)u(x)\}}\big(\frac{1}{|x|^\mu}\ast G(x,u)\big) \big(\frac{1}{\ell\theta} -\frac{1}{4\ell}\big)V(x)u^2.
\endaligned
$$
Since $2<\theta<4$, the conclusion follows.
\end{proof}
Next, we check that $\Phi$ verifies the Mountain Pass Geometry.
\begin{lem}\label{mountain:1}
Assume that $0<\mu<N$ and conditions $(f_1)-(f_3)$ hold. Then,
\begin{itemize}
  \item[$(1)$] \quad There exist $\rho, \delta_0>0$ such that $\Phi|_{S_\rho}\geq\delta_0>0$, $\forall u\in S_\rho=\{u\in E:\|u\|=\rho\}$.
  \item[$(2)$] \quad There is $r>0$ and $e\in H^{1}_0(B_{1})$ with $\|e\|>r$ such that $\Phi(e)< 0$.
\end{itemize}
\end{lem}
\begin{proof}
(1). By \eqref{Growth},
$$
|G(x,u)|\leq|F(u)|\leq c_0|u|^{\frac{2N-\mu}{N-2}},
$$
from which it follows that
$$
\aligned
\Phi(u) &\geq  \frac12\|u\|^2-C_1(\int_{\R^N}\big(\frac{1}{|x|^\mu}\ast |u|^{\frac{2N-\mu}{N-2}}\big) |u|^{\frac{2N-\mu}{N-2}},
\endaligned
$$
and so, by Hardy-Littlewood-Sobolev inequality,
$$
\aligned
\Phi(u) &\geq  \frac12\|u\|^2-C_1\|u\|^{\frac{2(2N-\mu)}{N-2}}.
\endaligned
$$
Since $\frac{(2N-\mu)}{N-2}>1$, the conclusion $(1)$ follows if we choose $\rho$ small enough.

(2) \quad  Fixing $u_0 \in H^{1}_0(B_{R_0})$ with $u_0^{+}(x)=\max\{u_0(x),0\} \not= 0$, we set
$$
\mathcal{A}(t)=\Psi(\frac{tu_0}{\|u_0\|})>0 \,\,\ \mbox{for} \,\,\, t>0.
$$
By direct calculus,
$$
\frac{\mathcal{A}'(t)}{\mathcal{A}(t)}\geq \frac{\theta}{t} \,\,\, \mbox{for all} \,\,\, t>0.
$$
Hence, integrating the above inequality over $[1, s\|u_0\|]$ with $s>\frac{1}{\|u_0\|}$,  we find
$$
\Psi(su_0)\geq \Psi(\frac{u_0}{\|u_0\|})\|u_0\|^{\theta} s^{\theta}.
$$
Therefore
$$
\Phi(su_0)\leq C_1 s^2-C_2s^{\theta} \,\,\, \mbox{for} \,\,\ s > \frac{1}{\|u_0\|},
$$
and conclusion $(2)$ holds for $e=s u_0$ with $s$ large enough.
\end{proof}
Applying the Mountain Pass theorem without (PS) condition \cite{MW}, we know that there exists a $(PS)_{c_V}$ sequence $(u_n) \subset E$ such that
$$
\Phi'(u_n)\to0 \,\,\, \mbox{and} \,\,\, \quad \Phi(u_n)\to
{c_V},
$$
where $c_V$ is the Mountain Pass level characterized by
$$
0<c_V:=\inf_{\gamma\in \Gamma} \max_{t\in [0,1]}
\Phi(\gamma(t))\leq\inf_{u\in E\backslash\{0\}} \max_{t\geq 0}
\Phi(tu)
$$
with
$$
\Gamma:=\{\gamma\in \mathcal{C}^1([0,1], E):\gamma(0)=0\ \ \hbox{and}   \ \  \Phi(\gamma(1))<0\}.$$
Hence, from the proof of Lemma \ref{mountain:1},
\begin{equation} \label{m2}
0<c_V\leq d:=\inf_{u\in H^1_0(B_{1})\backslash\{0\}} \max_{t\geq 0}
\tilde{\Phi}(tu),
\end{equation}
where
$$
\tilde{\Phi}(u)=\frac12\int_{B_{1}}(|\nabla u|^2+mu^2)-\frac{1}{2}\int_{B_{1}}\int_{B_{1}}\frac{F(u(x))F(u(y))}{|x-y|^\mu},
$$
and $m$ is given in (\ref{m}). Here, it is important to observe that $d$ is independent of the choice of $\ell$ and $R$.

\begin{lem}\label{BD}
Assume that $0<\mu<N$ and conditions $(f_1)-(f_3)$ hold. Then, the $(PS)_{c_V}$ sequence $(u_n)$
is bounded by a constant independent of the choice of $\ell$ and $R$.
\end{lem}
\begin{proof}
By Lemma \ref{AR2},
$$
\Phi(u_n)-\frac{1}{\theta}\Phi'(u_n)u_n\geq (\frac12-\frac{1}{\theta})\|u_n\|^2,
$$
which means $(u_n)$ is bounded in $E$. Moreover, we may assume that
$$
\|u_n\|^2\leq \frac{2\theta}{\theta-2}(c_V+1)\leq \frac{2\theta}{\theta-2}(d+1) \,\,\,\, \forall n \in \mathbb{N},
$$
showing the lemma, because $d$ is independent of the choice of $\ell$ and $R$.
\end{proof}

Before proving the next lemma, we need to fix some notations. In what follows,
\begin{equation} \label{B}
\mathcal{B}:=\Big\{u \in E:\|u\|^2\leq \frac{2\theta}{\theta-2}(d+1)\Big\}
\end{equation}
and
$$
K(u)(x):=\frac{1}{|x|^\mu}\ast G(x, u).
$$
With the above notations, we are able to show the ensuing estimate

\begin{lem}\label{BNT}
Assume $0<\mu<\frac{N+2}{2}$ and conditions $(f_1)-(f_3)$ holds. Then, there exists $\ell_0>0$, which is independent of $R$, such that
$$
\frac{\sup_{u\in\mathcal{B}}|K(u)(x)|_{L^{\infty}(\R^N)}}{\ell_0}<\frac12.
$$
\end{lem}
\begin{proof}
From the definition of $G$,
$$
|G(x, u)|\leq|F(u)|,
$$
and so,
$$
|G(x, u)|\leq c_0|s|^{2^*}\ \  \hbox{and}\ \ |G(x, u)|\leq c_0|s|^{p}.
$$
Thereby,
$$\aligned
|K(u)(x)|&=\Big|\int_{\R^N}\frac{G(x, u)}{|x-y|^\mu}dy\Big|\\
&\leq \Big|\int_{|x-y|\leq1}\frac{G(x, u)}{|x-y|^\mu}dy\Big|+\Big|\int_{|x-y|\geq1}\frac{G(x, u)}{|x-y|^\mu}dy\Big|\\
&\leq c_0\int_{|x-y|\leq1}\frac{|u|^{p}}{|x-y|^\mu}dy+c_0\int_{|x-y|\geq1}|u|^{2^*}dy\\
&\leq c_0\int_{|x-y|\leq1}\frac{|u|^{p}}{|x-y|^\mu}dy+C_1.
\endaligned
$$
Choosing $t=\frac{2^*}{p}>\frac{N}{N-\mu}$ , it follows from H\"{o}lder inequality,
$$\aligned
\int_{|x-y|\leq1}\frac{|u|^{p}}{|x-y|^\mu}dy&\leq \big(\int_{|x-y|\leq1}|u|^{2^*}dy\big)^{\frac{1}{t}}\big(\int_{|x-y|\leq1}\frac{1}{|x-y|^{\frac{t\mu}{t-1}}}\big)^{\frac{t-1}{t}}\\
&\leq  C_1\big(\int_{|r|\leq1}{|r|^{N-1-\frac{t\mu}{t-1}}}dr\big)^{\frac{t-1}{t}}.
\endaligned
$$
Once $N-1-\frac{t\mu}{t-1}>-1$, there is $C_2>0$ such that
$$
\int_{|x-y|\leq1}\frac{|u|^{p}}{|x-y|^\mu}\leq C_2 \quad \forall x  \in \mathbb{R}^N.
$$
From this, there exists $\ell_0>0$ verifying
\begin{equation} \label{PC}
\frac{\sup_{u\in\mathcal{B}}|K(u)(x)|_{L^{\infty}(\R^N)}}{\ell_0}\leq\frac12.
\end{equation}
\end{proof}
From now on, we take $\ell \geq \ell_0$ and consider the penalized problem with nonlinearity defined in \eqref{PNT}.

\begin{lem}\label{PPS1}
Assume that $0<\mu<\frac{N+2}{2}$ and conditions $(f_1)-(f_3)$ hold. Then,  the  $(PS)_{c_V}$ sequence $(u_n)$ satisfies the following property:  For each $\vr>0$ there exists $r=r(\vr)>R$ verifying
$$
\limsup_{n\to\infty}\int_{{\R}^N\backslash B_{2r}} (|\nabla u_n|^2+ V(x) |u_n|^2)<\vr.
$$
\end{lem}
\begin{proof}
From Lemma \ref{BD},
$$
\|u_n\|^2\leq \frac{2\theta}{\theta-2}(d+1) \quad \forall n \in \mathbb{N},
$$
where $d$ is independent of the choice of $\ell$ and $R$. Therefore, we can suppose that there exists $u\in E$ such that $u_n\rightharpoonup u$ in $E$.
Thus, for each $\vr>0$, there exists $r>R>0$, such that
\begin{equation} \label{EST1}
\omega^{\frac1N}_N\|u_n\|\big(\int_{r\leq |x|\leq 2r}|u|^{2^*}\big)^{\frac{1}{2^*}}<\frac{\vr}{8},
\end{equation}
where $\omega_N$ is the volume of the unitary ball in $\R^N$.

 Let $\eta_r\in C^{\infty}(B^c_r)$ be such that  $\eta_r(x)=1$ if $x\notin B_{2r}(0)$, with $0\leq \eta_r(x)\leq 1$ and $|\nabla \eta_r(x)|\leq \frac{2}{r}$. Note that
$$
\aligned
\int_{\R^N} \eta_r(|\nabla u_n|^2+ V( x) |u_n|^2)&=\Phi'(u_n)(u_n\eta_r)+\int_{\R^N}\big(\frac{1}{|x|^\mu}\ast G( x, u_n)\big)g( x, u_n)u_n\eta_r\\
&\hspace{4mm}-\int_{\R^N}u_n\nabla u_n\nabla\eta_r.
\endaligned
$$
Since $(u_n\eta_r)$ is bounded in $E$, it follows that $\Phi'(u_n)(u_n\eta_r)=o_n(1)$. Moreover, recalling that $\eta_r(x)=0$ in $B_R$,  we obtain
$$
\aligned
\int_{|x|\geq r}\eta_r(|\nabla u_n|^2&+ V(x) |u_n|^2)\\
& \leq \int_{|x|\geq r}\frac{\sup_{n}|{K}(u_n)(x)|_{L^{\infty}(\R^N)}}{\ell_0}\eta_rV(x)|u_n|^2-\frac{2}{r}\int_{r\leq|x|\leq 2r}u_n\nabla u_n+o_n(1).
\endaligned
$$
From Lemma \ref{BNT}, there holds
$$
\frac{{\displaystyle \sup_{n}}|{K}(u_n)(x)|_{L^{\infty}(\R^N)}}{\ell_0}\leq \frac12,
$$
and so,
\begin{equation} \label{EST2}
\int_{|x|\geq 2r}(|\nabla u_n|^2+ V(x)|u_n|^2)\leq \frac{4}{r}\int_{r\leq|x|\leq 2r}|u_n||\nabla u_n|+o_n(1).
\end{equation}
By H\"{o}lder's inequality,
$$
\int_{r\leq|x|\leq 2r}|u_n||\nabla u_n|\leq |\nabla u_n|_{L^2(\R^N)}\big(\int_{r\leq|x|\leq 2r}|u_n|^2\big)^{\frac12}.
$$
Since $u_n\to u$ in $L^2(r\leq|x|\leq 2r)$ and $(u_n)$ is bounded in $E$, it follows that
$$
\limsup_{n\to\infty}\int_{r\leq|x|\leq 2r}|u_n||\nabla u_n|   \leq {\|u_n\|}\big(\int_{r\leq|x|\leq 2r}|u|^2\big)^{\frac{1}{2}}.
$$
Using H\"{o}lder's inequality again, we get
$$
\limsup_{n\to\infty}\int_{r\leq|x|\leq 2r}|u_n||\nabla u_n|   \leq {\|u_n\|}\big(\int_{r\leq|x|\leq 2r}|u|^{2^*}\big)^{\frac{1}{2^*}}|B_{2r}|^{\frac1N}.
$$
Once $|B_{2r}|=\omega_N2^Nr^N$, by \eqref{EST1} and \eqref{EST2},
$$
\limsup_{n\to\infty}\int_{|x|\geq 2r}(|\nabla u_n|^2+ V(x)|u_n|^2)  \leq 8\omega^{\frac1N}_N\|u_n\|\big(\int_{r\leq|x|\leq 2r}|u|^{2^*}\big)^{\frac{1}{2^*}}<\vr.
$$
\end{proof}

\begin{lem}\label{COM}
Assume $0<\mu<\frac{N+2}{2}$ and conditions $(f_1)-(f_3)$ hold. Then, $\Phi$ satisfies the $(PS)_{c_V}$ condition.
\end{lem}
\begin{proof}
Since $u_n\rightharpoonup u$ in $E$, $ \Phi'(u_n)u_n=o_n(1)$ and $\Phi'(u_n)u=o_n(1)$, it follows that
$$
\|u_n-u\|^2=\int_{\R^N}K(u_n)g(x, u_n)(u_n-u)+o_n(1).
$$
Now, our goal is to show that the following limit holds,
$$
\int_{\R^N}K(u_n)g(x, u_n)(u_n-u)=o_n(1).
$$
We begin with recalling that by Lemma \ref{BNT}, there exists $C>0$ such that
$$
|K(u_n)|\leq C \quad \forall n \in \mathbb{N}.
$$
Notice that (\ref{Growth}) and $|g(x,s)s|\leq |s|^p$ with $p<2^*$,  for each $r>0$, the Sobolev's compact embedding implies
$$
\big|\int_{B_r}K(u_n)g(x, u_n)(u_n-u)\big|\leq C\int_{B_r}\big|g(x, u_n)(u_n-u)\big|\to0.
$$
From (\ref{Growth}), there also holds $|g(x,s)s|\leq |s|^{2^*}$, and so,
$$
\int_{\R^N\backslash B_r}K(u_n)|g(x, u_n)u_n|\leq C\int_{\R^N\backslash B_r}|u_n|^{2^*}.
$$
Combining Lemma \ref{PPS1} with Sobolev embedding theorem, given $\vr>0$, there exists $r(\vr)>0$ such that
$$
\limsup_{n\to \infty}\int_{\R^N\backslash B_r}K(u_n)|g(x, u_n)u_n|\leq C_1\vr.
$$
Similarly, applying H\"{o}lder inequality, we can also prove that
$$
\limsup_{n\to \infty}\int_{\R^N\backslash B_r}K(u_n)|g(x, u_n)u|\leq C_2\vr.
$$
In conclusion,
$$
\int_{\R^N}K(u_n)g(x, u_n)(u_n-u)\to 0.
$$
\end{proof}
Applying Lemmas \ref{mountain:1}  and \ref{COM},  we have the following result.
\begin{thm}\label{AQ}
Assume $0<\mu<\frac{N+2}{2}$ and conditions $(f_1)-(f_3)$ hold. Then, problem $(APE)$ has a positive solution with
$$
\|u\|^2\leq \frac{2\theta d}{\theta-2}.
$$
\end{thm}

In the following, we will study the $L^{\infty}$ estimate of the solution $u$. To this end, we adapt some techniques in \cite{AS, BK}.
\begin{lem}\label{LE}
Let $u$ be the solution obtained in  Theorem \ref{AQ}. Then, there exists a constant $M_0$ which depends only on $N, \mu, \theta, m, p, c_0$, such that
$$
|u|_{{\infty}}\leq M_0.
$$
\end{lem}
\begin{proof} By hypothesis, $u$  is a solution of
$$
-\Delta u +V(x)u=K(u)f(u) \,\,\, \mbox{in} \,\,\, \mathbb{R}^{N}
$$
with $K(u) \in L^{\infty}(\mathbb{R}^{N})$ and $|K(u)|_{\infty} \leq \frac{1}{2}$. Since $V(x) \geq 0$ for all $x \in \mathbb{R}^{N}$ and
$$
|f(t)| \leq c_0|t|^{2^{*}-1} \,\,\, \forall t \in \mathbb{R},
$$
it follows that
$$
-\Delta u \leq \frac{1}{2}a(x)(1+|u|) \,\, \mbox{in} \,\, \mathbb{R}^{N},
$$
where $a(x)=|u|^{\frac{4}{N-2}} \in L^{\frac{N}{2}}(\mathbb{R}^{N})$. By a Trudinger-Br\'ezis-Kato iteration argument, see Struwe's book \cite[ Lemma B3, pg 273]{St}, we deduce that $u \in L^{s}(\mathbb{R}^{N})$ for all $s >1$. Moreover, once $|a|_{\frac{N}{2}}$ does not depend of $R>0$, the norms $|u|_{s}$ also do not depend of $R>1$. Now, fixing $s$ large enough, the bootstrap arguments implies that there is $M_0>0$, which is  independent of $R$, such that
$$
|u|_{{\infty}}\leq M_0,
$$
finishing the proof.

\end{proof}

\section{Proof of Theorem 1.1}

 From Lemma \ref{AQ}, $(APE)$ has a positive solution $u_R\in E$ for each $R>1$. Thereby, in order to prove the existence of solution problem $(SNE)$, we must show that there is $R>1$ such that $u_R$ satisfies the inequality
$$
f(u_R)\leq \frac{V(x)}{\ell_0}u_R\ \ \hbox{for}\ \ |x|\geq R.
$$
In fact, let $v$ be the $C^{\infty}(\mathbb{R}^N\backslash\ \{0\})$ harmonic function
$$
v(x)=\frac{R^{N-2}|u|_{\infty}}{|x|^{N-2}}.
$$
By Lemma \ref{LE}, we have the inequality
$$
u\leq v \ \  \hbox{on}\ \  \partial B_{R},
$$
which implies that the function
$$
w=
\left\{\begin{array}{l}
(u-v)_{+},\ \  \hbox{if}\ \  |x|\geq R,\\
\displaystyle 0,\hspace{1.2cm}\ \  \hbox{if}\ \  |x|\leq R
\end{array}
\right.
$$
belongs to $D^{1,2}(\R^N)$. Since $\Delta v=0$ in $\R^N\backslash\ B_{R}(0) $, $w=0$ on $\partial B_{R}$ and $w\geq0$,  it follows that
$$
\aligned
\int_{\R^{N}}|\nabla w|^{2}&=\int_{\R^{N}}\nabla (u-v)\nabla w=\int_{|x|\geq R}\big(K(u)(x)g(x,u)w-V(x)uw\big)\\
& \leq \int_{|x|\geq R}\big(\frac{1}{\ell_0}K(u)(x)-1\big)V(x)uw\\
&\leq0,
\endaligned
$$
showing that $w\equiv0$, i.e. $u\leq v$ in $|x|\geq R$, equivalently,
$$
u(x)\leq\frac{R^{N-2}|u|_{\infty}}{|x|^{N-2}}\leq\frac{R^{N-2}M_0}{|x|^{N-2}}, \ \ \hbox{for  all}\ \  |x|\geq R.
$$
Using that $|f(s)|\leq c_0|s|^{q-1}$, we have
$$
\frac{f(u)}{u}\leq c_{0}|u|^{q-2}\leq c_{0}M^{q-2}_0 \frac{R^{(q-2)(N-2)}}{|x|^{(q-2)(N-2)}} \ \ \hbox{for all}\ \ |x|\geq R.
$$
Now, fix $R>1$ such that $\mathcal{V}(R)>0$. Then, the last inequality gives combine with definition to give
$$
\frac{f(u)}{u}\leq c_{0}|u|^{q-2}\leq \frac{ c_{0}l_0M^{q-2}_0V(x)}{l_0\mathcal{V}(R)}, \quad \mbox{for all} \quad |x|\geq R.
$$
Thus, setting the number
$$
\mathcal{V}_0= c_{0}l_0M^{q-2}_0,
$$
and if there is $R>1$ such that
$$
\mathcal{V}(R)>\mathcal{V}_0
$$
we obtain the desired result for the $R>1$ above fixed, that is,
$$
f(u_R)\leq \frac{V(x)}{\ell_0}u_R\ \ \hbox{for}\ \ |x|\geq R,
$$
finishing the proof.  $\Box$

\section{ Proof of Theorem \ref{MRS2} }

In the proof of Theorem \ref{MRS2}, we replace the space $D^{1,2}(\mathbb{R}^{N})$ by $D_{rad}^{1,2}(\mathbb{R}^{N})$ and consider
$$
E_{rad}=\left\{u\in D_{rad}^{1,2}(\R^N):   \ \int_{\R^N} V( x)|u|^2<\infty\right\}.
$$

In this case, due to the Hardy-Littlewood-Sobolev inequality, the energy function
$$
I(u)=\frac12\|u\|^2-\frac{(N-2)^{2}}{2(2N-\mu)^{2}}\int_{\R^N}\big(\frac{1}{|x|^\mu}\ast |u|^{\frac{2N-\mu}{N-2}}\big) |u|^{\frac{2N-\mu}{N-2}}
$$
is well defined and belongs to $\mathcal{C}^1(E,\R)$.
Now, repeating the same ideas of the previous sections, we can consider again the problem $(APE)$. Here, the reader is invited to check that functional $\Phi$ still verifies, with natural modifications,  the Lemmas 2.3, 2.4 and 2.5.

\vspace{0.5 cm}

\noindent {\bf Proof of Theorem \ref{MRS2}}

From Lemmas 2.3, 2.4, 2.5  and Theorem 2.9, $(APE)$ has a positive solution $u_R \in E_{rad}$. Thereby, in order to prove that it is indeed a solution of problem $(SNE)$, we must  show that there is $R>1$ such that $u_R$ satisfies the inequality
$$
f(u_R(x))\leq \frac{V(x)}{\ell_0}u_R(x)\ \ \hbox{for}\ \ |x|\geq R.
$$
We begin recalling that
$$
|u_R(x)| \leq \frac{C\| u_R\|_{D^{1,2}}}{|x|^{\frac{N-2}{2}}} \,\,\,\, \forall  x \in \mathbb{R}^{N} \setminus \{0\}
$$
and
$$
\|u_R\|^{2} \leq \frac{2\theta d}{\theta -2},
$$
from where it follows that
$$
\|u_R\|_{D^{1,2}} \leq A,
$$
where $A= \sqrt{\frac{2\theta d}{\theta -2}}$. Then,
$$
u_R(x)\leq\frac{CA}{|x|^{\frac{N-2}{2}}}, \ \ \hbox{for}\ \  |x|\geq R.
$$
Since $f(t)= |t|^{\frac{4-\mu}{N-2}}t $, we have
$$
\frac{f(u_R)}{u_R}= |u_R|^{\frac{4-\mu}{N-2}}\leq \frac{(CA)^{\frac{4-\mu}{N-2}}}{|x|^{\frac{4-\mu}{2}}} \ \ \hbox{for}\ \ |x|\geq R.
$$
Using the definition of $\mathcal{W}(R)$, it follows that
$$
\frac{f(u_R)}{u_R}= |u_R|^{\frac{4-\mu}{N-2}}\leq \frac{\ell_0(CA)^{\frac{4-\mu}{N-2}}V(x)}{\ell_0 \mathcal{W}(R)} \ \ \hbox{for}\ \ |x|\geq R.
$$
Fixing
$$
\mathcal{W}_0=\ell_0 (CA)^{\frac{4-\mu}{N-2}}>0
$$
and supposing the there is $R>1$ such that
$$
\mathcal{W}(R)> \mathcal{W}_0,
$$
then, for the above $R>1$, we can ensure that
$$
\frac{f(u_R)}{u_R}\leq \frac{V(x)}{\ell_0}\ \ \hbox{for}\ \ |x|\geq R,
$$
implying that $I'(u_R)=0$ in $E_{rad}$. Now, using the Principle of Symmetric Criticality due to Palais \cite{Palais}, we can  conclude that $I'(u_R)=0$ in $E$,  finishing the proof. $\Box$

\end{document}